\documentclass[a4paper,11pt]{amsart}
\usepackage[foot]{amsaddr}

\makeatletter
\@namedef{subjclassname@2020}{\textup{2020} Mathematics Subject Classification}
\makeatother

\usepackage[utf8]{inputenc}
\usepackage{graphicx}
\usepackage{cite}
\usepackage[english]{babel}
\usepackage{todonotes}
\usepackage{a4wide}
\usepackage{amsmath}
\usepackage{amsfonts}
\usepackage{amssymb}
\usepackage{amsthm}
\usepackage[initials]{amsrefs}
\usepackage{physics}
\usepackage[draft=false, kerning=true]{microtype}

\newtheorem{lem}{Lemma}
\newtheorem{thm}[lem]{Theorem}
\newtheorem{prop}[lem]{Proposition}

\numberwithin{equation}{section}
\theoremstyle{definition}

\theoremstyle{remark}

\usepackage{mathtools}
\DeclareMathOperator{\Li}{Li}

\DeclarePairedDelimiterX{\setm}[2]{\{}{\}}{#1\,\delimsize\vert\,\mathopen{}#2}
\let\abs\relax
\DeclarePairedDelimiter\abs{\lvert}{\rvert}%

\allowdisplaybreaks[3]

\newcommand{\Z}{\mathbb{Z}}

\newcommand{\R}{\mathbb{R}}
\newcommand{\C}{\mathbb{C}}
\newcommand{\Oh}{\mathcal{O}}

\title[A Local Limit Theorem for Integer Partitions into Small Powers]{A Local Limit Theorem for\\ Integer Partitions into Small Powers}

\author[G. F. Lipnik]{Gabriel F. Lipnik}
\address[G. F. Lipnik]{
  Graz University of Technology, Institute of Analysis and Number
  Theory, 8010~Graz, Austria}
\email{math@gabriellipnik.at}

\author[M. G. Madritsch]{Manfred G. Madritsch}
\address[M. G. Madritsch]{
  Université de Lorraine, IECL, CNRS, F-54000 Nancy, France}
\email{manfred.madritsch@univ-lorraine.fr}

\author[R. F. Tichy]{Robert F. Tichy}
\address[R. F. Tichy]{
  Graz University of Technology, Institute of Analysis and Number
  Theory, 8010~Graz, Austria}
\email{tichy@tugraz.at}



\date{\today}

\usepackage{hyperref}

\hypersetup{%
pdftitle={A Local Limit Theorem for Integer Partitions into Small Powers},
pdfsubject={},
pdfauthor={Gabriel F. Lipnik, Manfred G. Madritsch, Robert F. Tichy},
pdfkeywords={integer partitions, partition function, central limit theorem, saddle point method, Mellin transform}
hyperindex=true,
plainpages=false
}

\begin{document}

\begin{abstract}
  The investigation of partitions of integers plays an important role in
  combinatorics and number theory. Among the many variations, partitions into
  powers $0<\alpha<1$ were of recent interest. In the present paper we want to
  extend our considerations of the length of a random partition by providing a
  local limit theorem.
\end{abstract}

\maketitle

\section{Introduction}

A partition of a positive integer $n$ is the representation of this integer as
sum of positive integers
\[
  n=a_1+\cdots+a_m
\]
with $1\leq a_1\leq \cdots\leq a_m$. We denote by $p(n)$ the number of such
partitions of $n$ of arbitrary length $m$. The study of this function has a long
history in combinatorics and number theory (see Andrews
\cite{andrews1976:theory_partitions} and the references therein). Among the
first establishing an asymptotic formula for $p(n)$ are Hardy and
Ramanujan~\cite{hardy_ramanujan1918:asymptotic_formulae_in}. This formula was
extended to a complete asymptotic expansion by Rademacher
\cite{rademacher1937:partition_function_p}. Their approach uses properties of
elliptic modular functions.
Ingham~\cite{ingham1941:tauberian_theorem_partitions} developed a more
elementary approach (comparable to our method) for the asymptotic analysis of
certain partition problems.

A canonical generalization is to consider partitions in powers of integers. In
particular, we consider representations of the form
\[
  n=\lfloor a_1^\alpha\rfloor+\cdots+\lfloor a_m^\alpha\rfloor,
\]
where $\alpha>0$ and $1\leq a_1\leq \cdots\leq a_m$. For the case $\alpha\in\Z$ Roth and Szekeres \cite{roth_szekeres1954:some_asymptotic_formulae} could
provide an asymptotic formula. Gafni \cite{gafni_2016:power_partitions} used the
cirle method to establish a similar result and in recent papers Tenenbaum, Wu
and Li \cite{tenenbaum_wu_li2019:power_partitions_and} as well as Debruyne and
Tenenbaum \cite{debruyne_tenenbaum2020:saddle_point_method_for_partitions} used
the saddle point method to establish a complete asymptotic expansion for the
case of $\alpha\in\Z$ and real $\alpha\geq1$, respectively.

The case of partitions in square roots ($\alpha=\frac12$) is of recent interest.
For a positive integer $n$ let $f(n)$ denote the number of unordered
factorizations as products of integers greater than $1$. Furthermore we define
the set
\[
  \mathcal{F}(x)=\{m\vert m\leq x, m=f(n)\text{ for some $n$.}\}
\]
Then Balasubramanian and Luca
\cite{balasubramanian_luca2011:number_factorizations_integer} provided an upper
bound for $\left| \mathcal{F}(x)\right|$ by considering the number $p(n)$ of
partitions of $n$ with $\alpha=\frac12$.

Their result was refined by Chen and Li \cite{chen_li2015:square_root_partition}
and Luca and Ralaivaosaona \cite{luca_ralaivaosaona2016:explicit_bound_number}
to obtain the asymptotic formula
\[
  p_{\frac12}(n)\sim
Kn^{-8/9}\exp\biggl(\frac{6\zeta(3)^{1/3}}{4^{2/3}}n^{2/3}+\frac{\zeta(2)}{(4\zeta(3))^{1/3}}n^{1/3}\biggr),\]
where
\[
  K=\frac{(4\zeta(3))^{7/18}}{\pi
  A^2\sqrt{12}}\exp\biggl(\frac{4\zeta(3)-\zeta(2)^2}{24\zeta(3)}\biggr)
\]
and $A$ is the Glaisher--Kinkelin constant (\textit{cf.} \cite{finch2003:mathematical_constants}*{Section 2.15}). In subsequent
work Li and Chen~\cites{li_chen2016:r_th_root, li_chen2018:r_th_root} extended
the result to arbitrary powers $0<\alpha<1$ not being of the form $\alpha=1/k$
for a positive integer~$k$. Li and Wu~\cite{li_wu2021:k_th_root} (using similar
methods as Tenenbaum, Wu and Li~\cite{tenenbaum_wu_li2019:power_partitions_and})
considered the case of $\alpha=1/k$. Finally we want to mention Chern
\cite{Chern2021:square_root_partitions} who provided the asymptotics with
explicit constants for the partition function in the case $\alpha=\frac12$.

In the present paper we want to consider the length $m$ of a random
\emph{restricted} partition into $\alpha$-powers. A \emph{restricted} partition
of a positive integer $n$ into $\alpha$-powers (or $\alpha$-partition) is a representation of the form
\[
  n=\lfloor a_1^\alpha\rfloor+\cdots+\lfloor a_m^\alpha\rfloor,
\]
with $1\leq a_1<a_2<\cdots<a_m$. We denote by $p_{\alpha}(n)$ and
$q_{\alpha}(n)$ the number of \emph{unrestricted} and \emph{restricted}
$\alpha$-partitions of $n$. Furthermore we denote by $p_{\alpha}(n,m)$ and
$q_{\alpha}(n,m)$ the number of \emph{unrestricted} and \emph{restricted}
partitions of length $m$, respectively.

Among the first to consider $p_1(n,m)$ were Erd\H{o}s and Lehner
\cite{erdos_lehner1941:distribution_of_summands_in_partitions}. They
investigated a local limit theorem for $m$ close to the mean. Distinct parts were
first studied by Wilf \cite{wilf1983:problems_in_combinatorial_asymptotics} and
Goh and Schmutz \cite{goh_schmutz1995:distinct_part_sizes_in_partitions}
provided a central limit theorem. Schmutz
\cite{schmutz1994:part_sizes_of_partitions} extended their result to
multivariate cases using Meinardus' scheme (see Meinardus
\cite{meinardus1954:meinardus_scheme}). Hwang
\cite{hwang2001:limit_theorems_number}, also based on Meinardus' scheme, proved a
central and a local limit theorem with weaker necessary conditions on the
summands in the random partition.

Meinardus' scheme is a clever way to incorporate the different steps of the
saddle point method (\textit{cf.} Section VIII.3 of Flajolet and Sedgewick
\cite{flajolet_sedgewick2009:analytic_combinatorics}) in order to provide an
asymptotic function for the partition function $p_1(n)$ under some conditions on
the summands. One of the conditions involves the singularity analysis of the
associated Dirichlet generating function. In Meinardus' original work as well as
in the work by Hwang this function has only one simple pole on the real line.
Granovsky and Stark \cite{Granovsky_Stark2012:meinardus_multiple_singularities}
and Chern \cite{Chern2021:square_root_partitions} adapted the scheme to multiple
poles. Madritsch and Wagner \cite{madritsch_wagner2010:central_limit_theorem}
allowed infinitely equidistant simple poles along a vertical line in the complex
plane in their central limit theorem. Ralaivaosaona
\cite{ralaivaosaona2012:random_prime_partitions}, motivated by an open problem
in Hwang's paper \cite{hwang2001:limit_theorems_number}, considered partitions
into prime numbers, whose associated Dirichlet generating function has a
completely different singular behavior. In a previous work \cite{lipnik_madritsch_tichy:central_limit_theorem} we
established a central limit theorem and the aim of the present work is to extend
these considerations to provide a local limit theorem.

\section{Statement of results}

Let $0<\alpha<1$ be a fixed real number and set $\beta:=1/\alpha>1$.
% Thus in the sequel we omit the index
% $\alpha$ for $q_{\alpha}(n)$ and $q_{\alpha}(n,m)$, respectively. 
Furthermore let $\Pi(n)$ be the set of \emph{restricted} partitions of $n$ into
$\alpha$-powers and $q(n)=\left|\Pi(n)\right|$ be its cardinality. Furthermore
we call $m$ the length of the partition. Then we denote by $\Pi(n,m)$ the set of
partitions of length $m$ and $q(n,m)=\left|\Pi(n,m)\right|$ its cardinality,
respectively.

We take a closer look at the random variable $\varpi_n$ counting the number of
summands in a random partition of $n$ into $\alpha$-powers. In a recent work we
established the following central limit theorem.
\begin{thm}[{\cite{lipnik_madritsch_tichy:central_limit_theorem}}]
  \label{thm:clt}
  Let $0<\alpha<1$ and let $\varpi_n$ be the random variable counting
  the number of summands in a random restricted partition of $n$ into
  $\alpha$-powers. Then $\varpi_n$ is asymptotically normally
  distributed with mean $\mathbb{E}(\varpi_n)\sim\mu_n$ and variance
  $\mathbb{V}(\varpi_n)\sim\sigma_n^2$, i.e.,
  \[
    \mathbb{P}\left(\frac{\varpi_{n} - \mu_n}{\sigma_n}< x\right)
    =\frac{1}{\sqrt{2\pi}}\int_{-\infty}^x e^{-t^2/2}\dd t+o(1),
  \]
  uniformly for all $x$ as $n\to\infty$. Moreover the mean~$\mu_n$ and the
  variance~$\sigma_n^2$ are given by
  \begin{equation}\label{eq:mu}
    \mu_n = \sum_{k\geq1}\frac{g(k)}{e^{\eta k}+1} \sim c_{1}n^{1/(\alpha + 1)}
  \end{equation}
  and
  \begin{equation}\label{eq:sigma2}
    \sigma_n^2 = \sum_{k\geq1}\frac{g(k)e^{\eta k}}{(e^{\eta k}+1)^2}
      -\frac{\left(\sum_{k\geq1}\frac{g(k)ke^{\eta k}}{(e^{\eta k}+1)^2}\right)^2}
      {\sum_{k\geq1}\frac{g(k)k^2e^{\eta k}}{(e^{\eta k}+1)^2}}
      \sim c_{2}n^{1/(\alpha + 1)}
  \end{equation}
  where $\eta$ is the implicit solution of
  \[
    n=\sum_{k\geq1}\frac{k}{e^{\eta k}+1}.
  \]
\end{thm}

In the present work we consider the local behavior of the random variable
$\varpi_n$. 
\begin{thm}
  \label{thm:main}
  Let $0<\alpha<1$ and let $\varpi_n$ be the random variable counting the number
  of summands in a random partition of $n$ into $\alpha$-powers. Furthermore let
  $\mu_n$ and $\sigma_n^2$ be the mean and variance of $\varpi_n$, respectively.
  If $m=\mu_n+x\sigma_n$ with $x=o\left(\sigma_n\right)$, then 
  \[
    \mathbb{P}\left(\varpi_{n}=m\right)
    =\frac{e^{-x^2/2}}{\sqrt{2\pi}\sigma_n}
    \left(1+\Oh\left(\frac{\abs{x}+\abs{x}^3}{n^{\alpha/(2\alpha+2)}}\right)\right),
  \]
  uniformly in $x$.
\end{thm}

%\todo[inline]{We could put more details in the exponent: $e^{-x^2/2+F(x/\sigma_n)}$.}

\section{Overview of proof}
We fix $0<\alpha<1$ and set $\beta:=1/\alpha$. The probability distribution of
$\varpi_n$ is given by $\mathbb{P}(\varpi_n=m)=q(n,m)/q(n)$ for $1\leq n\leq m$.
Thus we need to calculate $q(n,m)$ as well as $q(n)$. A convenient way to
establish this is to consider the bivariate generating function
\[
  Q(z,u)=1+\sum_{m\geq1}\sum_{n\geq1}q(n,m)u^mz^n,
\]
where we have set $q(n,m)=0$ for $m>n$. By a double application of Cauchy's
integral formula we have
\begin{gather}\label{eq:Cauchy_Q}
  q(n,m)=[u^mz^n]Q(z,u)=\frac{1}{(2\pi i)^2}\int_{\abs{u}=e^{\rho}}\int_{\abs{z}=e^{-r}}
    \frac{Q(z,u)}{z^{n+1}u^{m+1}}\dd{z}\dd{u}.
\end{gather}

For our treatment we prefer a multiplicative representation of this function.
For an integer $k$ we define $g(k)=\#\{n\geq 1\colon \lfloor
n^\alpha\rfloor=k\}$. Then we may write (\textit{cf.} Lemma 1 of
\cite{lipnik_madritsch_tichy:central_limit_theorem})
\[
  Q(z,u)=\prod_{k\geq1}\left(1+uz^k\right)^{g(k)}
  =1+\sum_{n\geq1}q(n)\mathbb{E}(u^{\varpi_n})z^n.
\]
Moreover we replace $z$ and $u$ by $e^{-r-it}$ and
$e^{\rho+i\theta}$, respectively, and define
\[
  f(\tau,\sigma)=\log Q(e^{-\tau},e^{\sigma}),
\]
where $\tau=r+it$ and $\sigma=\rho+i\theta$. Then our double integral from
\eqref{eq:Cauchy_Q} transforms to 
\begin{equation}\label{eq:Cauchy_f}
  \begin{split}
    q(n,m)&=\frac{1}{(2\pi i)^2}\int_{\abs{u}=e^{\rho}}\int_{\abs{z}=e^{-r}}
      \frac{Q(z,u)}{z^{n+1}u^{m+1}}\dd{z}\dd{u}\\
    &=\frac{\exp(-m\rho+nr)}{(2\pi)^2}\int_{-\pi}^{\pi}\int_{-\pi}^{\pi}
      \exp(-mi\theta+int+f(r+it,\rho+i\theta))\dd{t}\dd{\theta}.
  \end{split}
\end{equation}

We need to analyze the function $f$ and its derivatives. Thus for integers $p,q\geq0$ we write 
\[
  f^{(p+q)}_{pq}(\tau,\sigma)=\frac{\partial^{p+q}f}{\partial^p\tau\partial^q\sigma}(\tau,\sigma)
\]
for short. In Section \ref{sec:analysis} we provide the necessary methods to
obtain the estimates for $f$ and its derivatives. 

Now we expand $f$ around $(t,\theta)=(0,0)$:
\begin{multline*}
  f(r+it,\rho+i\theta)=f(r,\rho)+it f'_{10}(r,\rho)+i\theta f'_{01}(r,\rho)\\
  - f''_{20}(r,\rho)\frac{t^2}{2}-f''_{11}(r,\rho)t\theta-f''_{02}(r,\rho)\frac{\theta^2}{2}+\Oh\left(\sup_{\substack{p+q=3\\0\leq t_0\leq t\\0\leq \theta_0\leq \theta}}f'''_{pq}(r+it_0,\rho+i\theta_0)t^p\theta^q\right)
\end{multline*}
The central idea of the saddle point method is to chose
$(r,\rho)\in\R^*_+\times\R$ in such a way that the linear terms in the $\exp$
function in \eqref{eq:Cauchy_f} cancel, \textit{i.e.} $n=-f'_{10}(r,\rho)$ and
$m=f'_{01}(r,\rho)$. In Section \ref{sec:saddle_point} we will show that such a
choice of $(r,\rho)$ is always possible and unique.

Next we set
\[
  t_n=r^{1+3\beta/7}\quad\text{and}\quad
  \theta_n=r^{3\beta/7}
\]
and split the double integral in \eqref{eq:Cauchy_f} into three parts
\begin{align*}
  (I)& & &\abs{\theta}\leq \theta_n & \abs{t}&\leq t_n,\\
  (II)& & &\abs{\theta}> \theta_n & \abs{t}&\leq t_n\quad\text{and}\\
  (III)& & 0\leq&\abs{\theta}\leq \pi & \abs{t}&> t_n.
\end{align*}
In Section \ref{sec:estimates_away} we consider the parts $(II)$ and $(III)$ and
show that in these cases $\abs{f(r+it,\rho+i\theta)-f(r,\rho)}$ is ``small''.
This allows us to estimate the corresponding integrals.

Then in Section \ref{sec:q_n_m} we put everything together to obtain an
asymptotic formula for $q(n,m)$. Since considering $q(n)$ instead of $q(n,m)$
corresponds to putting $\sigma=0$, by a similar but
shorter treatment we obtain an asymptotic formula for $q(n)$ in Section \ref{sec:q_n}.

Finally in Section \ref{sec:proof_of_llt} we combine the asymptotic formulas in
order to show the local limit theorem (Theorem \ref{thm:main}).

\section{Analysis of the function $f$}
\label{sec:analysis}

We start with the analysis of $f$ and its derivatives using Mellin transform.
For a given function $h$, its Mellin transform
$\mathcal{M}[h](s)=h^*(s)$ is given by
\[h^*(s)=\int_0^\infty h(t)t^{s-1}\dd{t}.\]
Applying this transform to our function $f$ yields
\[
  f^*(s,\sigma)=\int_0^\infty f(\tau,\sigma)\tau^{s-1}\dd \tau
  =D(s)Y(s,e^\sigma),
\]
where $D(s)$ and $Y(s,u)$ are the associated Dirichlet series and the Mellin
transform of $\log(1+ue^{-t})$, respectively, \textit{i.e.}
\[
  D(s)=\sum_{k\geq1}\frac{g(k)}{k^s}
  \quad\text{and}\quad
  Y(s,u)=\int_{0}^\infty\log(1+ue^{-t})t^{s-1}\dd t.
\]

The actual estimates follow by the converse mapping transform. To this end for
$\alpha<\beta$ two reals we denote by $\langle \alpha,\beta\rangle$ a vertical
stripe in the complex plane, \textit{i.e.}
\[
  \langle \alpha,\beta\rangle:=\{z\in \C\colon \alpha\leq \Re z\leq \beta\}.
\]
\begin{thm}[{\cite{flajolet_gourdon_dumas1995:mellin_transforms_and}*{Theorem
  4}}, Converse Mapping Theorem]\label{thm:converse_mapping}
  Let $f(x)$ be continuous in $(0,+\infty)$ with Mellin transform $f^*(s)$
  having a nonempty fundamental strip $\langle\alpha,\beta\rangle$. Assume that
  $f^*(s)$ admits a meromorphic continuation to the strip $\langle \gamma, \beta
  \rangle$ for some $\gamma<\alpha$ with a finite number of poles there, and is
  analytic on the vertical line $\Re(s)=\gamma$. Assume also that there exists a real
  number $\eta\in(\alpha,\beta)$ such that
  \begin{gather}\label{fgd:eq23}
    f^*(s)=\Oh(\abs{s}^{-r})
  \end{gather}
  with $r > 1$ as $\abs{s}\to\infty$ in $\gamma\leq\Re(s)\leq\eta$. If $f^*(s)$
  admits the Laurent expansion
  \begin{gather}\label{fgd:eq24}
    f^*\asymp \sum_{(\xi,k)\in A}d_{\xi,k}\frac{1}{(s-\xi)^k}
  \end{gather}
  for $s\in \langle\gamma,\alpha\rangle$, then an asymptotic expansion of $f(x)$
  at $0$ is given by
  \[f(x)=\sum_{(\xi,k)\in A}d_{\xi,k}\left(\frac{(-1)^{k-1}}{(k-1)!}x^{-\xi}(\log x)^{k-1}\right)+\Oh(x^{-\gamma}).\]
\end{thm}

For a successful application of this theorem we need the singularity analysis of
$D(s)$ and $Y(s,u)$, respectively. We start with $D(s)$ and recall that
$g(k)=\#\{n\geq1\colon \lfloor
n^\alpha\rfloor = k\}$. Thus we obtain
\[
  g(k)=\lceil (k+1)^\beta\rceil -\lceil k^\beta\rceil
  =\sum_{\nu=1}^{\lceil \beta-1\rceil}\binom{\beta}{\nu}k^{\beta-\nu}+\widetilde{g}(k),
\]
where $\beta=1/\alpha$ and $\widetilde{g}$ is a
uniformly bounded function. Thus we have
\[
  D(s)=\sum_{\nu=1}^{\lceil \beta-1\rceil}\binom{\beta}{\nu}\zeta(s-\beta+\nu)
  +\Oh\left(\zeta(s)\right),
\]
where $\zeta$ is the Riemann zeta function.

%Moreover we have the following bounds on $g(k)$.
\begin{lem}{\cite{li_chen2016:r_th_root}*{Lemma 3.3}}
  \label{lem:order_of_g}
  For $0<\alpha<1$ we have
  \[(\beta-1)k^{\beta-1}<g(k)<\beta 2^\beta k^{\beta-1},\]
  where $\beta=1/\alpha$.
\end{lem}
Thus
\[
  D(s)=\sum_{k\geq1}\frac{g(k)}{k^s}\asymp \zeta(s-\beta+1),
\]
where $A\asymp B$ means that there exist constants $\gamma_1,\gamma_2>0$ such
that $\gamma_1 A\leq B\leq \gamma_2 A$.

Since $\zeta(s)$ is not bounded for $\Im(s)\to\pm\infty$, which is necessary for
the application of Theorem \ref{thm:converse_mapping}, we
need to also analyze the function $Y(s,u)$ defined as the Mellin
transform of $t\mapsto (1+ue^{-t})$, \textit{i.e.}
\[
  Y(s,u)=\int_{0}^\infty \log\left(1+ue^{-t}\right)t^{s-1}\mathrm{d}t.
\]

% The following lemma tells us that $Y(s,u)$ decays exponentially for
% $\Im(s)\to\pm\infty$. Since $\zeta(s)$ only grows polynomally we may apply the
% converse mapping.
% \begin{lem}[{\cite[Lemma 1]{hwang2001:limit_theorems_number}}]
%   For each fixed $u$ lying in the cut-plane $\C\setminus]-\infty,-1]$, the
%   function $Y(s,u)$ can be meromorphically continued into the whole $s$-plane
%   with simple poles at $s=0,-1,-2,\ldots$. Moreover, $Y(s,u)$ satisfies the
%   estimate
%   \[\abs{Y(\sigma+it,u)}\ll e^{-(\pi/2-\varepsilon)\abs{t}}\]
%   for any $\varepsilon>0$ as $\abs{t}\to+\infty$, uniformly as $\sigma$ and $u$
%   are restricted to compact sets.
% \end{lem}

Therefore we take a closer look at the function $f$ . By
using the Taylor series expansion of $\log(1+x)$ around $x=0$ we get that
\[
  f(\tau,\sigma)=\sum_{k\geq1}g(k)\log(1+e^{\sigma-k\tau})
  =-\sum_{k\geq1}g(k)\sum_{\ell\geq1}\frac{(-e^\sigma)}{\ell}e^{-\ell k\tau}
\]
% \[
%   Y(s,u)=-\Gamma(s)\Li_{s+1}(-u)
%   =\Gamma(s)\sum_{j\geq 1}\frac{(-1)^{j+1}}{j^{s+1}}u^j.
% \]
Thus we have
\begin{align*}
  f^*(s,\sigma)=\int_0^\infty f(\tau,\sigma)\tau^{s-1}\dd{\tau}
  =-D(s)\Li_{s+1}(-e^\sigma)\Gamma(s),
  % &=-\sum_{\nu=1}^{\lceil\beta-1\rceil}\binom{\beta}{\nu}\zeta(s-\beta+\nu)\Li_{s+1}(-e^\sigma)\Gamma(s)
  % +\Oh\left(\zeta(s)\Li_{s+1}(-e^\sigma)\Gamma(s)\right).
\end{align*}
where
\[
  \Li_s(z)=\sum_{n\geq1}\frac{z^n}{n^s},
  \quad\text{and}\quad
  \Gamma(s)=\int_0^\infty e^{-t}t^{s-1}\dd{t}=\mathcal{M}\left[e^{-t}\right](s)
\]
are the polylogarithm and the Euler gamma
function, respectively.

Therefore the Mellin transform of $f$ is a combination of the Riemann zeta
function, the Euler gamma function and the polylogarithm. First the Riemann zeta
function is analytic in the whole complex plane with the exception of a simple
pole in $s=1$ and residue $1$. Furthermore it grows polynomially for
$\Im(s)\to\pm\infty$. Secondly the polylogarithm is analytic in the slit plane
$\C\setminus[1,+\infty[$ (\textit{cf.}
\cite{flajolet1999:singularity_analysis_and}*{Theorem 1}). It also grows
polynomially for $\Im(s)\to\pm\infty$.  Finally the Euler gamma function is
analytic in the whole complex plane except zero and the negative integers, where
it has simple poles. Its residue in $s=0$ is also $1$. By Stirling's formula
this function decays exponentailly for $\Im(s)\to\pm\infty$. This final property
allows us to apply the converse mapping (Theorem \ref{thm:converse_mapping}).

For the derivatives of $f$ the situation is not much different. Let $p,q\geq0$
be positive integers. Then we obtain for the derivatives of $f$ that
\[
  f^{(p+q)}_{pq}(\tau,\sigma)
  =(-1)^{p+1}\sum_{k\geq1}g(k)k^{p}\sum_{\ell\geq1}\frac{(-e^{\sigma})}{\ell^{1-p-q}}e^{-\ell k\tau}.
\]
Thus one easily checks that
\[
  f^*_{pq}(s,\sigma)
  =\int_0^\infty f^{(p+q)}_{pq}(\tau,\sigma)\tau^{s-1}\dd{\tau}
  =(-1)^{p+1}D(s-p)\Li_{s+1-p-q}(-e^{\sigma})\Gamma(s).
\]

Now we analyze $D(s)$. To this end we note that $g(k)$ is a sum of powers of $k$
(plus a rest). Thus for $p,q\geq0$ we recursively introduce the function
$h_{\gamma,p,q}\colon \R^2\to \R$ as follows:
\begin{align*}
  h_{\gamma,0,0}(\tau,\sigma)
  &=\sum_{k\geq1}k^{\gamma}\log(1+e^{\sigma-k\tau})
  =-\sum_{k\geq1}k^{\gamma}
  \sum_{\ell\geq1}\frac{(-1)^{\ell}}{\ell}e^{\ell(\sigma-k\tau)}\quad\text{and}\\
  h_{\gamma,p,q}(\tau,\sigma)
  &=\frac{\partial^{p+q}}{\partial\tau^p\partial\sigma^q}h_{\gamma,0,0}(\tau,\sigma)
  =-\sum_{k\geq1}k^{\gamma+p}
  \sum_{\ell\geq1}(-1)^{\ell+p}\ell^{p+q-1}e^{\ell(\sigma-k\tau)}.
\end{align*}
These definitions will come handy when we replace $g(k)$ by its order (see Lemma
\ref{lem:order_of_g}). The following lemma estimates the functions
$h_{\gamma,p,q}$ with $\gamma>0$ and $p,q\geq0$ non-negative integers.

% \begin{lem}[{\cite[Theorem 1]{flajolet1999:singularity_analysis_and}}]
%   The function $\Li_s(e^{-u})$ satisfies the singular expansion
%   \[\Li_s(e^{-u})\Gamma(1-s)u^{s-1}+\sum_{j\geq0}\frac{(-1)^j}{j!}\zeta(s-j)u^j\]
%   in the sector $-\pi+\varepsilon<\arg(1-e^{-u})<\pi-\varepsilon$.
% \end{lem}

\begin{lem}\label{lem:the_function_h} Let $\delta>0$. Then for $\rho\in\R$ and $r\to 0^+$ we have
  \[h_{\gamma,p,q}(r,\rho)=\Li_{\gamma+2-q}(-e^{\rho})\Gamma(\gamma+p+1)r^{-(\gamma+p+1)}+\Oh\left(r^{-\frac12}\right).\]
  %uniformly for $\delta<e^\rho<\delta^{-1}$.
  % Moreover for $r>0$ and $\rho\to\pm\infty$ we have
  % \[h_{\gamma,p,q}(r,\rho)\asymp \Li_{\gamma+2-q}(-e^{\rho})\Gamma(\gamma+p+1)r^{-(\gamma+p+1)}.\]
\end{lem}

\begin{proof}
  Note that $h_{\gamma,p,q}$ is a so called harmonic sum, \textit{i.e.} a sum of
  the form
  \[\sum_{k\geq1}\lambda_kh(\mu_kx,\rho)\]
  with real $\lambda_k,\mu_k$ for $k\geq1$.
  Therefore we denote by $H_{\gamma,p,q}$ the Mellin transform of
  $h_{\gamma,p,q}$ with respect to $\tau$. Thus
  \begin{align*}
    H_{\gamma,p,q}(s,\sigma)=\int_0^\infty h_{\gamma,p,q}(\tau,\sigma)\tau^{s-1}\dd{\tau}
    %&=\sum_{k\geq1}k^{\gamma+p}\sum_{\ell\geq1}(-1)^\ell \ell^{p+q-1}e^{\ell\sigma}\int_0^\infty e^{-k\ell\tau}\tau^{s-1}\dd{\tau}
    %&=\sum_{k\geq1}k^{\gamma+p-s}\sum_{\ell\geq1}(-1)^\ell \ell^{p+q-1-s}e^{\ell\sigma}\int_0^\infty e^{-\tau}\tau^{s-1}\dd{\tau}
    &=\zeta(s-p-\gamma)\Li_{s+1-p-q}(-e^{\sigma})\Gamma(s).
  \end{align*}

  In each vertical strip $a< \Re(s)< b$ the Riemann zeta function $\zeta(s)$ and
  the polylogarithm $\Li_s(z)$ grow polynomial in $\Im(s)$ whereas $\Gamma(s)$
  decays exponentially (Stirling's formula). Thus we may apply the converse
  mapping (Theorem \ref{thm:converse_mapping}). Since $\sigma\in\R$,
  $\gamma>0$ and $p\geq0$ the only pole with $\Re(s)>\frac12$ is the one of $\zeta(s-p-\gamma)$
  in $s=\gamma+p+1$. Finally we get by Theorem \ref{thm:converse_mapping} that
  \[h_{\gamma,p,q}(\tau,\sigma)=\Li_{\gamma+2-q}(-e^{\sigma})\Gamma(\gamma+p+1)\tau^{-(\gamma+p+1)}
  +\Oh\left(\tau^{-\frac12}\right).\]
  Plugging this in \eqref{eq:partial_derivative_of_f} yields the lemma.
\end{proof}

Our first application is the estimation of the derivatives of $f$.
  
\begin{lem}\label{lem:derivative_of_f}
  Let $p,q\geq0$ be integers. Then for reals $r$ and $\rho$ such that $r\to0^+$
  we have
  \begin{multline*}
    \frac{\partial^{p+q}f}{\partial \tau^p\partial
  \sigma^q}(r,\rho)\\
  =(-1)^{p+1}\sum_{\nu=1}^{\lceil\beta-1\rceil}\binom{\beta}{\nu}\Li_{\beta-\nu+2-q}(-e^{\sigma})\Gamma(\beta-\nu+p+1)r^{-(\beta-\nu+p+1)}
  +\Oh\left(r^{-(p+\frac12)}\right).
  \end{multline*}
\end{lem}

\begin{proof}
  Recall that
  % \[Q(z,u)=\prod_{k\geq1}\left(1+uz^k\right)^{g(k)}
  % \quad\text{and}\quad
  % g(k)=\sum_{\nu=1}^{\lceil\beta-1\rceil}\binom{\beta}{\nu}k^{\beta-\nu}+\widetilde{g}(k),\]
  % with $\max_{k\geq1}\left|\widetilde{g}(k)\right|\leq g_0$ being a bounded function.
  % Since $f(\tau,\sigma)=\log Q(e^{-\tau},e^{\sigma})$ we have
  % \begin{align*}
  %   f(\tau,\sigma)=\sum_{k\geq1}g(k)\log\left(1+e^{-k\tau+\sigma}\right)
  %   =-\sum_{k\geq1}g(k)\sum_{\ell\geq1}\frac{(-1)^\ell}{\ell}e^{\ell (\sigma-k\tau)}.
  % \end{align*}
  % Now we consider the derivative with respect to $\tau$ and $\sigma$. Then for
  % integers $p,q\geq0$ we obtain
  for integer $p,q\geq0$ we have
  \[
    f^{(p+q)}_{pq}(\tau,\sigma)
    =(-1)^{p+1}\sum_{k\geq1}g(k)k^{p}\sum_{\ell\geq1}\frac{(-e^{\sigma})}{\ell^{1-p-q}}e^{-\ell k\tau}.
  \]
  Together with the definition of $g(k)$ this yields
  \begin{equation}\label{eq:partial_derivative_of_f}
    \begin{split}
      f^{(p+q)}_{pq}(\tau,\sigma)
      =-\sum_{\nu=1}^{\lceil\beta-1\rceil}\binom{\beta}{\nu}h_{\beta-\nu,p,q}(\tau,\sigma)
      +\Oh\left(h_{0,p,q}(\tau,\sigma)\right).
    \end{split}
  \end{equation}
  The rest follows by applying Lemma \ref{lem:the_function_h}.
\end{proof}

\section{Uniqueness of the saddle point}
\label{sec:saddle_point}

In this section we use the functions $h_{\gamma,p,q}$ for showing that the
saddle point is in fact unique.
% $\abs{f(r+it,\rho+i\theta)-f(r,\rho)}$ for $(r,\rho)$ the saddle point.``large'' $\abs{t}$ or
% $\abs{\theta}$. In order to properly define, what we mean by ``large'' we have
% to introduce the saddle points $r$ and $\rho$. In particular, we set $r$ and
% $\rho$ such that for given positive integers $n$ and $m$ the following system is
% satisfied:
By the saddle point we mean the unique solution $(r,\rho)\in\R^*_+\times\R$ of the following system:
\begin{gather}\label{eq:saddle_point_system}
  \left\{
  \begin{aligned}
    n&=-f'_{10}(r,\rho)\quad\text{and}\\
    m&=f'_{01}(r,\rho).
  \end{aligned}
  \right.
\end{gather}

% \[Y(u,s)=\Li_{s+1}(-u)\Gamma(s).\]
% Thus
% \[\Li_{\beta+1}(-e^{\rho})\Gamma(\beta+1)
% =Y(e^{\rho},\beta)\frac{\Gamma(\beta+1)}{\Gamma(\beta)}=\beta
% Y(e^{\rho},\beta)=\frac{\rho^{\beta+1}}{\beta+1}\left(1+\Oh\left(\rho^{-2}\right)\right)\]

We establish existence and uniqueness in two steps. First we show that for every
$\rho\in\R$ there exists a unique $r=r(\rho)$ such that the first equation is
satisfied. This provides us with a relation between $n$ and $\rho$, which we use
in the second equation in order to establish the uniqueness of the solution of
the system \eqref{eq:saddle_point_system}.

As indicated above we first show that for fixed $\rho\in\R$ the first equation has a
solution.
\begin{lem}\label{lem:unique_r}
  Let $n\geq1$ be an integer. Then for fixed $\rho\in\R$ there exists a unique
  $r=r(\rho)>0$ such that
  \begin{gather}\label{n_is_f_tau}
    n=-f'_{10}(r,\rho)=\sum_{k\geq1}\frac{kg(k)}{e^{kr-\rho}+1}.
  \end{gather}
  Moreover we have
  \begin{gather}\label{eq:n_asymptotic_in_rho}
    n\asymp \Li_{\beta+1}(-e^{\rho})r^{-(\beta+1)}.
  \end{gather}
\end{lem}

Here again we mean by $A\asymp B$ that there exist two constants $\gamma_1$ and
$\gamma_2$ (only depending on $\beta=1/\alpha$) such that $\gamma_1 B\leq A\leq \gamma_2B$.

% \todo[inline]{Should we use Lemma \ref{lem:Li+Gamma_asymptotic} to get
% an order without $\Li$ and $\Gamma$?}

\begin{proof}
  %We have 
  %\[n=-f_\tau(r,\rho)=\sum_{k\geq1}\frac{kg(k)}{e^{kr-\rho}+1}.\]
  Fixing $\rho\in\R$ we get that for each $k\geq1$ the function
  \[r\mapsto\frac{kg(k)}{e^{kr-\rho}+1}\]
  is strictly decreasing and therefore the sum over the $k$ is strictly
  decreasing and it exists a unique $r=r(\rho)>0$ such that
  \eqref{n_is_f_tau} is satisfied.

  Moreover by Lemma \ref{lem:order_of_g} we have
  \[n=\sum_{k\geq1}\frac{kg(k)}{e^{kr-\rho}+1}
  \asymp
  \sum_{k\geq1}\frac{k^\beta}{e^{kr(\rho)-\rho}+1}=h_{\beta-1,1,0}(r,\rho),\]
  which together with Lemma \ref{lem:the_function_h} proves the asymptotic
  formula \eqref{eq:n_asymptotic_in_rho}.
\end{proof}

After showing that for each $n$ and $\rho$ there exists a unique $r$ we need to
show that if we vary $\rho$ then we obtain all possible $m$. To this end we
define
\begin{gather}\label{eq:S}
  S(\rho)=f'_{01}(r(\rho),\rho)=\sum_{k\geq1}\frac{g(k)}{e^{kr(\rho)-\rho}+1},
\end{gather}
where $r=r(\rho)$ is the implicit function from Lemma \ref{lem:unique_r}. In the proof of
the uniqueness of the solution we need that the function $S(\rho)$ is
increasing. This reduces to showing that the Hessian determinant
$\delta(r,\rho)$ of $f$ is positive. Thus we define
\[
  \delta(r,\rho):=f''_{20}(r,\rho)f''_{02}(r,\rho)-f''_{11}(r,\rho)^2
\quad (r>0, \rho\in\R).
\]
% where we simply wrote
% \[f^{(\ell)}_{jk}(r,\rho):=\frac{\partial^\ell f(r,\rho)}{\partial r^j\partial \rho^k}\quad(j+k=\ell).\]

\begin{lem}\label{lem:hessian}
  For $(r,\rho)\in  \R_+^*\times \R$ we have $\delta(r,\rho)>0$.
\end{lem}

\begin{proof}
  From the definition of $f$ we have
  \begin{align*}
    \delta(r,\rho)
    &=f''_{20}(r,\rho)f''_{02}(r,\rho)-f''_{11}(r,\rho)^2\\
    &=\left(\sum_{k\geq1}k^2h(k)\right)\left(\sum_{k\geq1}h(k)\right)
      -\left(\sum_{k\geq1}kh(k)\right)^2,
  \end{align*}
  where, for short, we have set
  \[h(k)=\frac{g(k)e^{kr-\rho}}{(e^{kr-\rho}+1)^2}>0.\]
  The lemma now follows by applying the Cauchy-Schwarz-inequality.
  % with
  % \[a_k=k\sqrt{h(k)}\quad\text{and}\quad b_k=\sqrt{h(k)}.\]
\end{proof}

\begin{lem}[{\cite{hwang2001:limit_theorems_number}*{Lemma 7}}]
  \label{lem:Li+Gamma_asymptotic}
  Let $\gamma>0$. Then
  \[
    \Li_{\gamma}(-u)=
    -\frac{(\log u)^{\gamma}}{\Gamma(\gamma+1)}\left(1+\Oh\left((\log u)^{-2}\right)\right)
  \]
  as $\abs{u}\to+\infty$ in the sector $\abs{\arg u}\leq \pi-\varepsilon$.
\end{lem}

Now we have all the tools in hand to show that the system
\eqref{eq:saddle_point_system} has a unique solution.

\begin{prop}
  For each $m,n$ such that $1\leq m\leq (M_0-\varepsilon)n^{\beta/(\beta+1)}$
  the system \eqref{eq:saddle_point_system} has a unique solution $(r,\rho)\in
  \R_+^*\times \R$, where $M_0$ is a constant only depending on $\beta$.
\end{prop}

As we will see below the maximal length of a restricted partition is
$\asymp n^{\beta/(\beta+1)}$. Therefore the range of $m$ fits our picture.

\begin{proof}
  First we show that $S'(\rho)>0$. Since $\rho$ is a solution of the equation
  $n=-f'_{10}(r,\rho)$ we get by implicit differentiation that
  $r'(\rho)=-f''_{11}(r,\rho)/f''_{20}(r,\rho)$. Since
  \begin{gather}\label{eq:f_20>0}
    f''_{20}(r,\rho)=\sum_{k\geq1}\frac{k^2g(k)e^{kr-\rho}}{(e^{kr-\rho}+1)^2}>0
  \end{gather}
  for all $(r,\rho)\in\R^*_+\times\R$ we get by Lemma \ref{lem:hessian} that
  \[S'(\rho)=f''_{11}
  r'+f''_{02}=-\frac{(f''_{11})^2}{f''_{20}}+f''_{02}=\frac{\delta(r,\rho)}{f''_{20}(r,\rho)}>0.\]

  Now we consider the image of $S$. To this end we note that by Lemma \ref{lem:order_of_g} we have
  \[
    S(\rho)=\sum_{k\geq1}\frac{g(k)}{e^{kr-\rho}+1}
    \asymp \sum_{k\geq1}\frac{k^{\beta-1}}{e^{kr-\rho}+1}
    = h_{\beta-1,0,1}(r,\rho).
  \]

  Thus an application of Lemma \ref{lem:the_function_h} together with the
  asymptotic order for $n$ in \eqref{eq:n_asymptotic_in_rho} yields
  \begin{equation}\label{eq:S_rho_asymp}
    S(\rho)\asymp -\Li_{\beta}(-e^\rho)r^{-\beta}
    \asymp -\Li_{\beta}(-e^\rho)\left(\frac{n}{-\Li_{\beta+1}(-e^\rho)}\right)^{\frac{\beta}{\beta+1}}.
  \end{equation}

  Now we take a closer look on the right side of this asymptotic order. We note
  that for $\abs{u}\leq 1$ and $\Re s> 0$ we have
  \[
    \Li_{s}(-u)=\sum_{j\geq1}\frac{(-u)^{j}}{j^s}=-u+\Oh(u^2).
  \]
  Thus for $\rho\to-\infty$ we get for the right side of \eqref{eq:S_rho_asymp} that
  \[
    -\Li_{\beta}(-e^\rho)\left(\frac{n}{-\Li_{\beta+1}(-e^\rho)}\right)^{\frac{\beta}{\beta+1}}
    =e^{\rho\frac{1}{\beta+1}}n^{\frac{\beta}{\beta+1}}\left(1+\Oh\left(e^{2\rho}\right)\right),
  \]
  implying
  \[
    \lim_{\rho\to-\infty}S(\rho)=0.
  \]

  For $\rho\to+\infty$ we apply Lemma \ref{lem:Li+Gamma_asymptotic} and obtain
  for the right side of \eqref{eq:S_rho_asymp} that there exists a constant $M_0$
  (only depending on $\beta$) such that
  \[
    -\Li_{\beta}(-e^\rho)\left(\frac{n}{-\Li_{\beta+1}(-e^\rho)}\right)^{\frac{\beta}{\beta+1}}=M_0n^{\frac{\beta}{\beta+1}}\left(1+\Oh\left(\rho^{-2}\right)\right),
  \]
  providing the upper bound.
\end{proof}

\section{Estimates away from the positive real line}
\label{sec:estimates_away}

After applying Cauchy's integral formula two times in \eqref{eq:Cauchy_f} we
need to analyze the function $f(r+it,\rho+i\theta)$. In the above section we
considered how to choose optimal values for $(r,\rho)\in\R^*_+\times\R$. The aim
of this section is show that $f(r+it,\rho+i\theta)$ may be replaced by
$f(r,\rho)$ if $r+it$ and/or $\rho+i\theta$ are sufficiently far away from the
real line. Therefore we have to estimate $\abs{f(r+it,\rho+i\theta)-f(r,\rho)}$
if $\abs{t}$ and/or $\abs{\theta}$ is ``large''. We will first explain what we
mean by ``large'' by setting
\[
  t_n=r^{1+3\beta/7}\quad\text{and}\quad \theta_n=r^{3\beta/7}.
\]

Our first lemma considers the case that $\abs{t}$ is large regardless of $\abs{\theta}$.
\begin{lem}\label{lem:integrant_estimate_for_large_t}
  Let $\delta>0$ and let $r$, $t$, $\rho$ and $\theta$ be reals such that
  $r^{1+3\beta/7}<\abs{t}\leq \pi$ and $0\leq \abs{\theta}\leq \pi$. Then we have 
  \[
    \frac{\abs{Q(e^{-(r+it)},e^{\rho+i\theta})}}{Q(e^{-r},e^\rho)}
    \leq
    \exp\left(-\frac{2e^\rho(\beta-1)}{(1+e^\rho)^{2}}c_3\left(\frac14 r^{-\beta/7}\right)\right).
  \]
\end{lem}

\begin{proof}
  We start following the lines of proof of Lemma 5 of Luca and Ralaivaosaona
  \cite{luca_ralaivaosaona2016:explicit_bound_number}. Thus we have
  \begin{align*}
    \abs{1+e^{\rho+i\theta}e^{-k(r+it)}}^{2}
    &=\abs{1+e^{\rho-kr+i(\theta-kt)}}^{2}
    =(1+e^{\rho-kr+i(\theta -kt)})(1+ e^{\rho-kr-i(\theta -kt)})\\
    &=1+e^{\rho-kr+i(\theta -kt)}+ e^{\rho-kr-i(\theta -kt)}+e^{2\rho-2kr}\\
    &=1+2 e^{\rho-kr}+(e^{\rho-kr})^2-2e^{\rho-kr}+2 e^{\rho-kr}\cos(\theta -kt)\\
    &=(1+ e^{\rho-kr})^2-2e^{\rho-kr}(1-\cos(\theta-kt)).
  \end{align*}
  Plugging this into the above ratio yields
  \begin{align*}
    \biggl(\frac{\abs{Q(e^{-(r+it)}, e^{\rho+i\theta})}}{Q(e^{-r},e^\rho)}\biggr)^2
    &=\prod_{k\geq1}\biggl(1-\frac{2e^{\rho-kr}(1-\cos(\theta-kt))}{(1+ e^{\rho-kr})^2}\biggr)^{g(k)}\\
    &=\exp\biggl(\sum_{k\geq1}g(k)\log\biggl(1-\frac{2 e^{\rho-kr}(1-\cos(\theta-kt))}{(1+e^{\rho-kr})^2}\biggr)\biggr)\\
    &\leq\exp\biggl(-\sum_{k\geq1}g(k)\frac{2 e^{\rho-kr}(1-\cos(\theta-kt))}{(1+e^{\rho-kr})^2}\biggr)\\
    &\leq\exp\biggl(-\frac{2e^\rho}{(1+e^\rho)^2}\sum_{k\geq1}g(k)e^{-kr}(1-\cos(\theta-kt))\biggr)\\
    &\leq\exp\biggl(-\frac{2e^\rho}{(1+e^\rho)^2}\sum_{k\geq1}k^{\beta-1}e^{-kr}(1-\cos(\theta-kt))\biggr).
  \end{align*}

  Now we focus on the sum in the exponent. To this end we follow the lines of
  proof of Lemma~2.5 of Li and Chen \cite{li_chen2018:r_th_root}. Thus we note
  Gauss' product formula for the Gamma function (see Chapter 2 of Remmert
  \cite{remmert1998:classical_topics_in}):
  % \[\Gamma(t)=\lim_{n\to\infty} \frac{n!n^t}{t(t+1)\cdot (t+n)}.\]
  % Thus
  \begin{align*}
    \Gamma(\beta)
    &=\lim_{n\to\infty}\frac{n!n^\beta}{\beta(\beta+1)\cdots (\beta+n)}\\
    &=\lim_{n\to\infty}\frac{(n-1)!n^{\beta-1}}{\beta(\beta+1)\cdots (\beta+n-2)}.
  \end{align*}
  The limit implies that there exists a constant $c_3$ such that
  \[ c_3\leq \frac{(n-1)!n^{\beta-1}}{\beta(\beta+1)\cdots (\beta+n-2)}\]
  or otherwise said
  \[ n^{\beta-1}\geq c_3\frac{\beta(\beta+1)\cdots (\beta+n-2)}{(n-1)!}.\]

  Since
  \[\sum_{n\geq1}\frac{\beta(\beta+1)\cdot (\beta+n-2)}{(n-1)!}z^n
    =\frac{z}{(1-z)^\beta},\quad\abs{z}<1,\]
  we obtain
  \begin{align*}
    &\sum_{k\geq1}k^{\beta-1}e^{-kr}(1-\cos(\theta-kt))\\
    &\quad\geq c_3\left(\sum_{k\geq1}\frac{\beta(\beta+1)\cdot (\beta+k-2)}{(k-1)!}e^{-kr}
      -\Re\left(\sum_{k\geq1}\frac{\beta(\beta+1)\cdot (\beta+k-2)}{(k-1)!}e^{-kr+i(\theta-kt)}\right)\right)\\
    &\quad=c_3\left(\frac{e^{-r}}{(1-e^{-r})^\beta}
      -\Re\left(\frac{e^{-r-it+i\theta}}{(1-e^{-r-it})^\beta}\right)\right)\\
    &\quad\geq c_3\left(\frac{e^{-r}}{(1-e^{-r})^\beta}
    -\frac{e^{-r}}{\left|1-e^{-r-it}\right|^\beta}\right).
  \end{align*}

  For the rest we follow the proof of Lemma 2.6 in Li and
  Chen~\cite{li_chen2018:r_th_root}. Starting with the second denominator, we obtain
  \begin{align*}
    \abs{1 - e^{-r - it}}^{\beta} &= (1 - 2e^{-r}\cos t + e^{-2r})^{\beta/2}\\
                                  &= ((1 - e^{-r})^{2} +2e^{-r}(1 - \cos t))^{\beta/2}\\
    &\geq (1-e^{-r})^{\beta}\Bigl(1 + \frac{2e^{-r}}{(1 - e^{-r})^{2}}(1 - \cos t_{n})\Bigr)^{\beta/2}
  \end{align*}
  for $t_{n}\leq \abs{t}\leq \pi$.

  Noting that
  \begin{align*}
    1 - \cos t_{n} = \frac12 t_{n}^{2} + \Oh(t_{n}^{4}) = \frac12 r^{2+6\beta/7}+ \Oh(r^{4 + 12\beta/7})
  \end{align*}
    and that $e^{-r} = 1 - r + \Oh(r^{2})$, we find
  \begin{align*}
    \frac{2e^{r}}{(1-e^{r})^{2}}(1 - \cos t_{n}) = \frac{2(1 + \Oh(r))(\frac12 r^{2 + 6\beta/7} + \Oh(r^{4+12\beta/7}))}{r^{2}(1 + \Oh(r))} = r^{6\beta/7}(1 + \Oh(r)).
  \end{align*}
  Plugging this into the first denominator implies
  \begin{align*}
    \abs{1 - e^{-r - it}}^{\beta} &\geq (1 - e^{-r})^{\beta}\bigl(1 + r^{6\beta/7}(1 + \Oh(r))\bigr)^{\beta/2}\\
    &= (1 - e^{-r})^{\beta}\Bigl(1 + \frac{\beta}{2}r^{6\beta/7} + \Oh(r^{6\beta/7 + 1})\Bigr).
  \end{align*}
  Thus
  \begin{align*}
    \frac{e^{-r}}{(1-e^{-r})^\beta}-\frac{e^{-r}}{\left|1-e^{-r-it}\right|^\beta}
    &\geq \frac{e^{-r}}{(1-e^{-r})^\beta}\left(1-\frac{1}{1+\frac{\beta}2r^{6\beta/7}+\Oh\left(r^{6\beta/7+1}\right)}\right)\\
    &=\frac{e^{-r}}{(1-e^{-r})^\beta}\left(\frac{\beta}2r^{6\beta/7}+\Oh\left(r^{6\beta/7+1}\right)\right)\\
    &=\frac{1+\Oh(r)}{r^\beta(1-\Oh(r))}\left(\frac{\beta}2r^{6\beta/7}+\Oh\left(r^{6\beta/7+1}\right)\right)\\
    &=\frac{\beta}2 r^{6\beta/7-\beta}+\Oh\bigl(r^{6\beta/7+1-\beta}\bigr)
  \end{align*}
  and finally
  \[
    \frac{e^{-r}}{(1-e^{-r})^\beta}-\frac{e^{-r}}{\left|1-e^{-r-it}\right|^\beta}
    \geq \frac14 r^{6\beta/7-\beta}
  \]
  for sufficiently small~$r$. So we consequently obtain
  \begin{equation*}
    \biggl(\frac{\abs{Q(e^{-(r+iy)}, e^{\rho+i\theta})}}{Q(e^{-r}, e^\rho)}\biggr)^2 \leq
    \exp\biggl(-\frac{2e^\rho(\beta-1)}{(1+e^\rho)^{2}}c_3\Bigl(\frac14 r^{6\beta/7-\beta}\Bigr)\biggr),
  \end{equation*}
  completing the proof.
\end{proof}

This estimate suffices for calculating an asymptotic for $q(n)$. For an
asymptotic formula for $q(n,m)$, however, we also need the following estimate if
$\abs{t}$ is small and $\abs{\theta}$ is large.

\begin{lem}\label{lem:integrant_estimate_for_small_t}
  Let $\delta>0$ and let $r$, $t$, $\rho$ a $\theta$ be reals such that
  $\abs{t}\leq r^{1+3\beta/7}$ and $r^{3\beta/7}<\abs{\theta}\leq \pi$. Then we have
  \[
    \frac{\abs{Q(e^{-(r+it)},e^{\rho+i\theta})}}{Q(e^{-r},e^\rho)}
    \leq
    \exp\biggl(-c_7\frac{e^\rho}{(1+e^\rho)^{2}}r^{-\beta/7}\biggr),
  \]
  where $c_7$ depends only on $\beta$.
\end{lem}

\begin{proof}
  Following the same lines as at the beginning of Lemma
  \ref{lem:integrant_estimate_for_large_t} above we get that
  \begin{align*}
    \biggl(\frac{\abs{Q(e^{-(r+it)}, e^{\rho+i\theta})}}{Q(e^{-r},e^\rho)}\biggr)^2
    &\leq\exp\biggl(-\frac{2e^\rho}{(1+e^\rho)^2}\sum_{1/(3r)\leq k\leq 1/(2r)}k^{\beta-1}e^{-kr}(1-\cos(\theta-kt))\biggr)\\
    &\leq\exp\biggl(-\frac{2e^\rho}{(1+e^\rho)^2}\left(1-\cos(\frac12 r^{3\beta/7})\right)e^{-1/2}\sum_{1/(3r)\leq k\leq 1/(2r)}k^{\beta-1}\biggr).
  \end{align*}
  First for the cosine part we note the following inequality
  \[
    1-\cos t\geq\frac{2}{\pi^2}t^2\quad\text{for }\abs{t}\leq\pi.
  \]
  Secondly we use Euler's summation formula (\textit{cf.} Theorem 3.2
  in Apostol \cite{apostol1976:introduction_to_analytic}) to obtain
  \[\sum_{1/(3r)\leq k\leq 1/(2r)}k^{\beta-1}\ll r^{-\beta},\]
  thus proving the lemma.
\end{proof}

% Finally we need an estimate for completing the tails. This is standard, however,
% since it is a single line we present it here for completeness.

% \begin{lem}\label{lem:completing_tails}
%   Let $a,t>0$. Then
%   \[\int_{t}^{\infty}\exp\left(-\frac{a^2}2x^2\right)\dd{x}
%   =\frac{\exp\left(-\frac{a^2}2t^2\right)}{ta^2}\]
% \end{lem}

% \begin{proof}
%   We have
%   \begin{align*}
%     \int_{t}^{\infty}\exp\left(-\frac{a^2}2x^2\right)\dd{x}
%     \leq\int_{t}^{\infty}\frac{x}{t}\exp\left(-\frac{a^2}2x^2\right)\dd{x}
%     %&=\int_{t^2}^{\infty}\frac{1}{2t}\exp\left(-\frac{a^2}2u\right)\dd{u}\\
%     =\frac{\exp\left(-\frac{a^2}2t^2\right)}{ta^2}
%   \end{align*}
% \end{proof}

\section{The number of partitions of $n$ of length $m$}
\label{sec:q_n_m}

%Flajolet and Sedjewick p.613ff.

Recall that $q(n,m)$ denotes the number of
partitions of $n$ of the form
\[n=\lfloor a_1^\alpha\rfloor+\cdots+\lfloor a_m^\alpha\rfloor.\]
Furthermore recall that by twice applying Cauchy's integral formula together
with the definition of the functions $Q$ and $f$ we obtain
\begin{multline*}
  q(n,m)=[u^mz^n]Q(z,u)=\frac{1}{(2\pi i)^2}\int_{\abs{u}=e^{\rho}}\int_{\abs{z}=e^{-r}}
  \frac{Q(z,u)}{z^{n+1}u^{m+1}}\dd{z}\dd{u}\\
    =\frac{\exp\left(-m\rho+nr\right)}{(2\pi)^2}\int_{-\pi}^{\pi}\int_{-\pi}^{\pi}
    \exp(-mi\theta+int+f(r+it,\rho+i\theta))
    \dd{t}\dd{\theta}.
\end{multline*}
We split the area $0\leq t,\theta\leq \pi$ of the double integral up into three parts:
\begin{align*}
  (I)& & &\abs{\theta}\leq \theta_n=r^{3\beta/7} & \abs{t}&\leq t_n=r^{1+3\beta/7},\\
  (II)& & &\abs{\theta}> \theta_n=r^{3\beta/7} & \abs{t}&\leq t_n=r^{1+3\beta/7}\quad\text{and}\\
  (III)& & 0\leq&\abs{\theta}\leq \pi & \abs{t}&> t_n=r^{1+3\beta/7}.
\end{align*}

For $(II)$ and $(III)$ we use Lemma \ref{lem:integrant_estimate_for_small_t} and
Lemma \ref{lem:integrant_estimate_for_large_t}, respectively. Thus we obtain
\begin{multline*}
  \left(\iint_{(II)}+\iint_{(III)}\right)
  \exp(-mi\theta+int+f(r+it,\rho+i\theta))
  \dd{t}\dd{\theta}\\
  \ll\exp\left(-m\rho+nr+f(r,\rho)-\frac{c_8e^\rho}{(1+e^\rho)^{2}}r^{-\beta/7}\right),
\end{multline*}
where $c_8$ is a combination of the constants in the Lemmas
\ref{lem:integrant_estimate_for_large_t} and \ref{lem:integrant_estimate_for_small_t}.

Now we concentrate on the first part and start with the inner integral with
respect to $t$:
\[J=\int_{-t_n}^{t_n}\exp(int+f(r+it,\rho+i\theta))\dd{t}.\]

Expanding $f(r+it,\rho+i\theta)$ around $t=0$ yields (using Lemma \ref{lem:derivative_of_f})
\begin{multline*}
  f(r+it,\rho+i\theta)\\
  =f(r,\rho+i\theta)+itf'_{10}(r,\rho+i\theta)
-\frac{t^2}{2}f''_{20}(r,\rho+i\theta)
-i\frac{t^3}{6}f'''_{30}(r,\rho+i\theta)
+\Oh\left(t^4r^{-(\beta+4)}\right).
\end{multline*}

Since $f''_{20}(r,\rho)>0$ (see \eqref{eq:f_20>0}) we may set
\begin{gather}\label{eq:B}
  B^2=B(r,\rho)^2=f''_{20}(r,\rho).
\end{gather}
Then we carry out the change of variables $t=v/B$. Thus using the relation of
$n$ and $r$ (see Equation \ref{eq:n_asymptotic_in_rho}) above we obtain
\begin{multline*}
  J=\frac{\exp(f(r,\rho+i\theta))}{B}\\
    \int_{-Bt_n}^{Bt_n}
    \exp\left(i\mathcal{Y}v
    -\frac{v^2}{2B^2}f''_{20}(r,\rho+i\theta)
    -i\frac{v^3}{6B^3}f'''_{30}(r,\rho+i\theta)
    +\Oh\left(r^{\beta}v^4\right)\right)\dd{v},
\end{multline*}
where we have written
\[
  \mathcal{Y}=\mathcal{Y}(r,\rho,\theta)=\frac{f'_{10}(r,\rho+i\theta)-f'_{10}(r,\rho)}{B(r,\rho)}
\]
for short.

Now expanding $f''_{20}(r,\rho+i\theta)$ around
$\theta=0$ we obtain
\[\frac{1}{B^2}f''_{20}(r,\rho+i\theta)
  =1+i\theta\frac{1}{B^2}f'''_{21}(r,\rho)+\Oh(\theta^2).
\]
Plugging this in $J$ we get
\begin{gather*}
  J=\frac{\exp(f(r,\rho+i\theta))}{B}
    \int_{-Bt_n}^{Bt_n}
    \exp\left(i\mathcal{Y}v-\frac{v^2}{2}\right)\left(1+R_1\right)\dd{v},
\end{gather*}
where
\[R_1=
  -i\frac{v^2\theta}{2B^2}f'''_{21}(r,\rho)
  -i\frac{v^3}{6B^3}f'''_{30}(r,\rho+i\theta)
  +\Oh\left(\theta^2v^2+r^{\beta} v^4\right).
\]

Completing the tails we obtain
\begin{align*}
  J&=\frac{\exp(f(r,\rho+i\theta))}{B}
\int_{-\infty}^{\infty}
\exp\left(i\mathcal{Y}v-\frac{v^2}{2}\right)\left(1+R_1\right)\dd{v}\\
&\quad+\Oh\left(\frac1{B^2t_n}\exp\left(\Re f(r,\rho +i\theta)-\frac12 B^2 t_n^2\right)\right).
\end{align*}

We may interpret the integral together with the powers of $v$ appearing in
$(1+R_1)$ as a Fourier transform of $v^L\exp(v^2/2)$ with $L\geq0$ an integer.
Then by induction on $L$ we obtain that
\[\frac{1}{2\pi}\int_{-\infty}^{\infty}e^{\mathcal{Y}iv-v^2/2}v^L\dd{v}
  =\frac{e^{-\mathcal{Y}^2/2}}{\sqrt{2\pi}}\sum_{0\leq 2\ell\leq L}
  \binom{L}{2\ell}\frac{(2\ell)!}{2^\ell \ell!}(i\mathcal{Y})^{L-2\ell}.
\]
Thus
\[
  J=\frac{\exp(f(r,\rho+i\theta)-\frac{\mathcal{Y}^2}{2})}{\sqrt{2\pi}B}
  \left(1+R_2\right)+\Oh\left(R_3\right),
\]
where
\begin{align*}
  R_2&=\frac{i\theta}{2B^2}\left(\mathcal{Y}^2-1\right)
  f'''_{21}(r,\rho)
  +\frac{1}{6B^3}\left(\mathcal{Y}^3-3\mathcal{Y}\right)
  f'''_{30}(r,\rho)\\
  R_3&=\left(\theta^2\abs{\mathcal{Y}}^2+r^{\beta} \abs{\mathcal{Y}}^4\right)
  \exp\left(\Re f(r,\rho+i\theta)-\Re\mathcal{Y}^2/2\right)/B\\
  &\quad+\frac1{B^2t_n}\exp\left(\Re f(r,\rho +i\theta)-\frac12 B^2 t_n^2\right)
\end{align*}

Now we return to the outer integral:
\[
  q(n,m)=\frac{e^{-m\rho+nr}}{(2\pi)\sqrt{2\pi}B}\int_{-\theta_n}^{\theta_n}\exp(-im\theta+f(r,\rho+i\theta)-\frac{\mathcal{Y}^2}{2})
  \left(1+R_2\right)+\Oh\left(R_3\right)\dd{\theta}
\]

Expanding $f(r,\rho+i\theta)$ and $f'_{10}(r,\rho+i\theta)$ around
$\theta=0$ yields
\begin{align*}
  \exp\left(-im\theta+f(r,\rho+i\theta)-\frac{\mathcal{Y}^2}{2}\right)
  =\exp\left(f(r,\rho)-b^2\frac{\theta^2}{2}\right)\left(1+R_4\right),
\end{align*}
where
\begin{align}
  b^2&=\left(f''_{02}(r,\rho)
  -\frac{\left(f''_{11}(r,\rho)\right)^2}
  {f''_{20}(r,\rho)}\right)\quad\text{and}\label{eq:b}\\
  R_4&=-i\left(f'''_{03}(r,\rho)
  -3\frac{f''_{11}(r,\rho)
    f'''_{12}(r,\rho)}
    {f''_{20}(r,\rho)}\right)\frac{\theta^3}6
  +\Oh\left(r^{-\beta}\theta^4\right).\notag
\end{align}

Putting everything together we get
\begin{equation}\label{eq:q_n_m}
  q(n,m)=\frac{e^{-m\rho+nr+f(r,\rho)}}{2\pi Bb}\left(1+\Oh\left(r^{2\beta/7}\right)\right),
\end{equation}
where $B$ and $b$ are as in \eqref{eq:B} and \eqref{eq:b}, respectively.

\section{The number of partitions of $n$}
\label{sec:q_n}

Recall that by $q(n)$ we denote the number of representations of $n$ of the form
\[n=\lfloor a_1^\alpha\rfloor+\cdots+\lfloor a_m^\alpha\rfloor\] with integers
$1\leq a_1<a_2<\cdots<a_s$ and $m\geq1$. Thus an application of Cauchy's
integral formula yields for $r>0$
\begin{align*}
  q(n)=[z^n]Q(z,1)&=\frac{1}{2\pi i}\oint_{\abs{z}=e^{-r}}\frac{Q(z,1)}{z^{n+1}}\dd{z}\\
  &=\frac{e^{nr}}{2\pi}\int_{-\pi}^{\pi}\exp\left(int+f(r+it,0)\right)\dd{t}.
\end{align*}

Since there is no $m$ in this case we suppose that $\rho=0$ and therefore choose
$r=r(0)$, where $r$ is the implicit function defined in Lemma \ref{lem:unique_r}.
Similar to above we split the last integral into two parts according to the size
of $t_n=r^{1+3\beta/7}$. In particular,
\[
  q(n)=I_1+I_2,
\]
where
\begin{align*}
  I_1&=\frac{e^{nr}}{2\pi}\int_{\abs{t}\leq t_n}\exp\left(int+f(r+it,0)\right)\dd{t}\quad\text{and}\\
  I_2&=\frac{e^{nr}}{2\pi}\int_{t_n<\abs{t}\leq\pi}\exp\left(int+f(r+it,0)\right)\dd{t}.
\end{align*}

We start our considerations with the second integral providing us with the
error term. By the symmetry of the integrand we have
\begin{align*}
  \abs{I_2}&\leq \frac{e^{nr+f(r,0)}}{2\pi}\int_{t_n}^{\pi}\frac{\abs{Q(e^{-r-it},1)}}{Q(e^{-r},1)}\dd{t}
\end{align*}
and an application of Lemma \ref{lem:integrant_estimate_for_large_t} yields
\begin{align*}
  \abs{I_2}&\ll \exp\left(nr+f(r,0)-c_4r^{-\beta/7}\right).
\end{align*}

Now we turn our attention to the main part $I_1$. We expand $f(r+it,0)$ around
$t=0$, \textit{i.e.}
\[
  f(r+it,0)=
  f(r,0)+itf'_{10}(r,0)
  -\frac{t^2}{2}f''_{20}(r,0)
  +\Oh\left(t^3\sup_{0\leq\abs{t_0}\leq t_n}f'''_{30}(r+it_0,0)\right).
\]
% Since for $k\geq1$ the function
% $r\mapsto (e^{rk}+1)^{-1}$ is decreasing, we may choose $r$ to be the unique solution 
% of
% \[n=-f'_{10}(r,0)=\sum_{k\geq1}\frac{kg(k)}{e^{rk}+1}.\]
% Clearly $r\to 0^+$ for $n\to+\infty$. Furthermore an application of Lemma
% \ref{lem:derivative_of_f} yields
% \begin{gather}\label{eq:n_asymptotic}
%   n=-f'_{10}(r,0)=-\sum_{\nu=1}^{\lceil\beta-1\rceil}\binom{\beta}{\nu}\Li_{\beta-\nu+2}(-1)\Gamma(\beta-\nu+2)r^{-\beta+\nu-2}+\Oh(r^{-3/2}).
% \end{gather}

By our choice of $\rho$ an application of Lemma \ref{lem:derivative_of_f} yields
\[
  n=-f'_{10}(r(0),0)=\sum_{\nu=1}^{\lceil\beta-1\rceil}\binom{\beta}{\nu}\Li_{\beta-\nu+2}(-1)\Gamma(\beta-\nu+2)r^{-\beta+\nu-2}+\Oh(r^{-3/2}).
\]
Furthermore we obtain for $B$ as in \eqref{eq:B} that
\[B^2=B(r(0),0)^2=f''_{20}(r,0)=-\sum_{\nu=1}^{\lceil\beta-1\rceil}\binom{\beta}{\nu}\Li_{\beta-\nu+3}(-1)\Gamma(\beta-\nu+3)r^{-\beta+\nu-3}+\Oh(r^{-5/2}).\]

Finally an application of Lemma \ref{lem:derivative_of_f} yields for the third
derivative (using that $t_n=r^{1+3\beta/7}$)
\[\sup_{0\leq\abs{t_0}\leq t_n}f'''_{30}(r+it_0,0)\ll r^{-(\beta+4)}.\]

Plugging everything in $I_1$ we obtain
\[I_1=\frac{e^{nr+f(r,0)}}{2\pi}\left(\int_{-t_n}^{t_n}\exp\left(-\frac{B^2}2t^2\right)\dd{t}\right)\left(1+\Oh\left(r^{2\beta/7}\right)\right).\]

Completing the tails and putting everything
together yields
\begin{equation}\label{eq:q_n}
  q(n)=\frac{e^{nr(0)+f(r(0),0)}}{\sqrt{2\pi}B}\left(1+\Oh\left(r(0)^{2\beta/7}\right)\right),
\end{equation}
where $r(\rho)$ is the implicit function defined in Lemma \ref{lem:unique_r} and
$B$ is defined in \eqref{eq:B}.

\section{The proof of the local limit theorem}
\label{sec:proof_of_llt}

%\todo[inline]{Continue here ...}

% For the proof of Theorem \ref{thm:main} we need to show that
% \begin{gather}\label{eq:llt}
%   \sup_{x\in\R}\left|\mathbb{P}(\varpi_n=\lfloor
%   \mu_n+x\sigma_n\rfloor)-\frac1{\sqrt{2\pi}}e^{-x^2/2}\right|\leq
%   \varepsilon_n.
% \end{gather}

We assume that $x$ is chosen such that $m=\mu_n+x\sigma_n$ is an integer. Then
combining \eqref{eq:q_n_m} and \eqref{eq:q_n} we have
\[
  \mathbb{P}(\varpi_n=m)=\frac{q(n,m)}{q(n)}=L_n(\rho)e^{H_n(\rho)}\left(1+\Oh\left(n^{-2\beta/(7\beta+7)}\right)\right),
\]
where
\begin{align}
  L_n(\rho)&=\frac{B(r(0),0)}{\sqrt{2\pi}B(r(\rho),\rho)b(r(\rho),\rho)}\label{eq:L_n}
  \intertext{and}
  H_n(\rho)&=-m\rho+nr(\rho)+f(r(\rho),\rho)-nr(0)-f(r(0),0).\label{eq:H_n}
\end{align}

Before we start our considerations we need some information on the derivatives
of $r(\rho)$, which will appear in the sequel. Hence, we note that
$n=-f'_{10}(r(\rho),\rho)$. Thus by implicit differentiation we obtain
\begin{align*}
  r'(\rho)&=-\frac{f''_{11}(r,\rho)}{f''_{20}(r,\rho)},\\
  r''(\rho)&=\frac{-f'''_{30}(r,\rho)f''_{11}(r,\rho)^2+2f'''_{21}(r,\rho)f''_{11}(r,\rho)f''_{20}(r,\rho)-f'''_{12}(r,\rho)f''_{20}(r,\rho)^2}{f''_{20}(r,\rho)^3}
\end{align*}
and
\begin{align*}
  r'''(\rho)&=-f^{(4)}_{13}(r,\rho) f''_{20}(r,\rho)^{-1} \\
  &\quad + \left(3 f^{(4)}_{22}(r,\rho) f''_{11}(r,\rho)+3 f'''_{12}(r,\rho)
    f'''_{21}(r,\rho)\right) f''_{20}(r,\rho)^{-2} \\
  &\quad -\left(3f^{(4)}_{31}(r,\rho) f''_{11}(r,\rho)^2+6
    f'''_{21}(r,\rho)^2 f''_{11}(r,\rho) 
    +3f'''_{12}(r,\rho) f'''_{30}(r,\rho) f''_{11}(r,\rho)\right) f''_{20}(r,\rho)^{-3} \\
  &\quad +\left(f^{(4)}_{40}(r,\rho) f''_{11}(r,\rho)^3+9 f'''_{21}(r,\rho)
    f'''_{30}(r,\rho) f''_{11}(r,\rho)^2\right) f''_{20}(r,\rho)^{-4} \\
&\quad -3 f''_{11}(r,\rho)^3 f'''_{30}(r,\rho)^2 f''_{20}(r,\rho)^{-5}.
\end{align*}
Using the estimates for the partial derivatives of $f$ in Lemma
\ref{lem:derivative_of_f} we get for integers $p,q\geq0$ that
\[
  f^{(p+q)}_{pq}(r,\rho)\asymp r^{-(\beta+p)}\asymp n^{\frac{\beta+p}{\beta+1}}.
\]
Plugging this into our derivatives of $r$ we get that $r'$, $r''$ and $r'''$ are
all a $\Oh\left(n^{-1/(\beta+1)}\right)$.

Now we start our considerations with the exponent $H_n$. To this end we write
$T(\rho)=f(r(\rho),\rho)$ for short. Recall that $n=-f'_{10}(r(\rho),\rho)$ and
$m=f'_{01}(r(\rho),\rho)=S(\rho)$ as defined in \eqref{eq:S}. Thus
\begin{align*}
  T'(\rho)&=-n r'(\rho)+S(\rho),\\
  T''(\rho)&=f''_{20}(r(\rho),\rho) r'(\rho)^2+f''_{11}(r(\rho),\rho) r'(\rho)
  -n r''(\rho)+S'(\rho)
\end{align*}
and
\begin{align*}
  T'''(\rho)&=f'''_{30}(r(\rho),\rho)r'(\rho)^3+2f'''_{21}(r(\rho),\rho)r'(\rho)^2
    +3f''_{20}(r(\rho),\rho)r'(\rho)r''(\rho)\\
  &\quad+f'''_{12}(r(\rho),\rho)r'(\rho)+2f''_{11}(r(\rho),\rho)r''(\rho)
    +n r'''(\rho)+S''(\rho).
\end{align*}
Then we have
\begin{align*}
  r(\rho)-r(0)&=r'(0)\rho+r''(0)\frac{\rho^2}{2}+\Oh\left(n^{-1/(\beta+1)}\rho^3\right)\\
  T(\rho)-T(0)&=T'(0)\rho+T''(0)\frac{\rho^2}{2}+\Oh\left(n^{\beta/(\beta+1)}
  \rho^3\right)\\
  S(\rho)&=S(0)+S'(0)\rho+\Oh\left(n^{\beta/(\beta+1)}
  \rho^2\right)
\end{align*}
Plugging this into \eqref{eq:H_n} yields
\begin{align*}
  H_n(\rho)
  &=-(S(0)+S'(0)\rho)\rho
    +n\left(r'(0)\rho+r''(0)\frac{\rho^2}{2}\right)
    +T'(0)\rho+T''(0)\frac{\rho^2}{2}+\Oh\left(n^{\beta/(\beta+1)} \rho^3\right)\\
  &=\left(-S'(0)+f''_{20}(r(0),0)r'(0)^2+f''_{11}(r(0),0)r'(0)\right)\frac{\rho^2}{2}
    +\Oh\left(n^{\beta/(\beta+1)} \rho^3\right)\\
  &=-S'(0)\frac{\rho^2}{2}
    +\Oh\left(n^{\beta/(\beta+1)} \rho^3\right).
\end{align*}

Since $S(0)=\mu_n$ and $S'(0)=\sigma_n^2$, we get from our choice of $x$ that
\[\frac{x}{\sigma_n}=\frac{m-\mu_n}{\sigma_n^2}=\frac{S(\rho)-S(0)}{S'(0)}.\]
Using Lagrange inversion (\textit{cf.} Chapter A.6 of Flajolet and Sedgewick
\cite{flajolet_sedgewick2009:analytic_combinatorics}) finally yields
\[
  \rho=\sum_{j\geq1}\rho_j\left(\frac{x}{\sigma_n}\right)^j
\]
with
\[
  \rho_j=\frac{1}{j}[w^{j-1}]\left(\frac{S(w)-S(0)}{S'(0)w}\right)^{-j}.
\]

Noting that $\rho_1=1$ we may replace $\rho$ by $x/\sigma_n$ and get that
\[
  H_n(\rho)=-\frac{x^2}{2}+\Oh\left(n^{\beta/(\beta+1)}\left(\frac{\lvert x\rvert}{\sigma_n}\right)^3\right).
\]

By a similar calculation we get for \eqref{eq:L_n} that
\[
  L_n(\rho)=\frac{1}{\sqrt{2\pi}\sigma_n}\left(1+\Oh\left(\frac{\abs{x}}{\sigma_n}\right)\right).
\]

Since $\sigma_n^2\sim n^{\beta/(\beta+1)}$ by Theorem \ref{thm:clt}, putting the
last two estimates together establishes Theorem \ref{thm:main}.

\section*{Acknowledgment}
The first author is supported by the Austrian Science Fund (FWF), project
W\,1230. The second author is supported by project ANR-18-CE40-0018 funded by
the French National Research Agency. The third author is supported by the
Austrian Science Fund (FWF), project F\,5510-N26 within the Special Research
Area “Quasi-Monte Carlo Methods: Theory and Applications” and project I\,4406-N.

Major Parts of the present paper were established when the three authors ware
mutually visiting each other. These visits were not only at their home
institutions in Graz and Nancy, but also at the Department of Mathematics at the
University of Klagenfurt, Austria, and the Centre International de Rencontres
Mathématiques in Marseille, France. The authors thank all the institutions for
their hospitality.

%\bibliography{literatur}


\end{document}